\documentclass[11pt]{amsart}
\pdfoutput=1

\usepackage[utf8]{inputenc}

\usepackage[margin=3cm]{geometry}

\usepackage{amssymb, amsmath, amsthm, amsfonts}
\usepackage{color}
\usepackage{mathrsfs}
\usepackage{physics}
\usepackage{mathtools}
\usepackage{tikz}
\usepackage{faktor}
\usepackage{tikz-cd}
\usepackage{bm}
\usepackage{hyperref}
\usepackage{csquotes}
\usepackage{stmaryrd}
\usepackage{yhmath}
\usepackage{enumitem}

\DeclareRobustCommand{\SkipTocEntry}[5]{}

\numberwithin{equation}{section}

\newtheorem{theorem}{Theorem}[section]
\newtheorem*{theorem*}{Theorem}
\newtheorem{lemma}[theorem]{Lemma}
\newtheorem{corollary}[theorem]{Corollary}
\newtheorem{definition}[theorem]{Definition}

\newtheorem{prop}[theorem]{Proposition}

\newtheorem*{conjecture*}{Conjecture}
\newtheorem*{idea*}{Main idea}
\theoremstyle{remark}
\newtheorem{remark}[theorem]{Remark}
\theoremstyle{remark}

\theoremstyle{remark}
\newtheorem*{notation*}{Notation}
\theoremstyle{remark}
\newtheorem*{conventions*}{Conventions}

\usepackage[UKenglish]{isodate}
\usepackage[UKenglish]{babel}

\newcommand{\R}{\mathbf{R}}

\newcommand{\C}{\mathbf{C}}
\newcommand{\Hom}{\mathrm{Hom}}

\newcommand{\id}{\text{id}}
\newcommand{\I}{\mathcal{I}}

\renewcommand{\O}{\mathscr{O}}

\newcommand{\Spec}{\mathrm{Spec}\,}

\newcommand{\T}{\mathscr{T}}
\newcommand{\D}{\mathscr{D}}
\newcommand{\M}{\mathcal{M}}
\newcommand{\F}{\mathscr{F}}

\newcommand{\Sp}{\mathrm{Sp}}

\newcommand{\Mscr}{\mathscr{M}}

\newcommand{\J}{\mathscr{J}}

\newcommand{\widehatotimes}{\,\,\widehat{\otimes}\,\,}

\newcommand{\Coh}{\mathbf{Coh}}
\newcommand{\coh}{\mathrm{coh}}
\newcommand{\Spf}{\mathrm{Spf}\,}
\DeclareMathOperator{\Ext}{\mathrm{Ext}}

\newenvironment{thmnum}
 {\begin{enumerate}[label=\upshape(\arabic*),ref=\thetheorem.\arabic*]}
 {\end{enumerate}}

\makeatletter
\newcommand{\sbullet}{} 
\DeclareRobustCommand\sbullet{%
  \mathord{\mathpalette\sbullet@{0.75}}%
}
\newcommand{\sbullet@}[2]{%
  \vcenter{\hbox{\scalebox{#2}{$\m@th#1\bullet$}}}%
}
\makeatother

\newcommand{\newabstract}[1]{%
  \par\bigskip
  \csname otherlanguage*\endcsname{#1}%
  \csname captions#1\endcsname
  \item[\hskip\labelsep\scshape\abstractname.]
}

\title{Completions of complexes of differential modules on singular schemes}

\author{Bruno Bori{\'c}}
\address{Department of Mathematics, Faculty of Natural Sciences, University of Zagreb, Bijeni{\v c}ka Cesta 30, 10000 Zagreb, Croatia}
\email{bboric.math@pmf.hr}

\author{Dalton A R Sakthivadivel}
\address{Department of Mathematics, CUNY Graduate Centre, 365 Fifth Avenue, New York, NY 10016}
\email{dsakthivadivel@gc.cuny.edu}

\date{\today}

\subjclass[2020]{14B05, 14F10, 58H10}

\begin{document}

\begin{abstract}

Spencer cohomology theory studies the cohomology of chain complexes of modules over the ring of differential operators $\mathscr{D}$ of a smooth analytic space. In this paper we give a generalisation of Spencer cohomology suitable for singular schemes of finite type over a field. Our motivation was a conjecture of Vinogradov concerning the homological properties of differential operators on singular affine varieties; namely, that complexes of certain such operators are acyclic if and only if the variety is smooth. We will provide a negative answer to Vinogradov's conjecture as stated. In principle Vinogradov's conjecture can also be posed for the Spencer complex of a general $\mathscr{D}$-module---however the answer is trivial, since singularities prohibit a definition of Spencer cohomology with any good properties. Our main result will be the construction of a Spencer complex on a large class of singular schemes which is suitable as a cohomology theory for the space. Following this we are able to ask the same question as Vinogradov in this case, where we give a more positive answer. Our main technique draws from Hartshorne's construction of de Rham cohomology by formal completion.

\end{abstract}

\maketitle


\section{Introduction}

Associated to any smooth analytic space $X$ and vector bundle $E$ on it are invariants coming from the chain complex
\begin{equation}\label{jet-complex-eq}
    0 \to E \to \J^r_X(E) \to \T^*_X \otimes_{\O_X} \J_X^r(E) \to \ldots \to \bigwedge^n \T^*_X \otimes_{\O_X} \J^r_X(E)\to 0,
\end{equation}
introduced by Spencer in \cite{spencer1969overdetermined}. We will refer to this as the {\it topological Spencer complex}. This followed a general interest in constructing algebraic invariants from the jet bundle of a smooth manifold representing the obstructions to solubility of overdetermined systems partial differential equations on the spaces, beginning with \cite{spencer1962deformation} and the 1964 thesis of Quillen \cite{quillen1964formal}. This cohomology turns out to be a very generic receptacle for cohomology theories which resolve some sheaf of constant functions, wherein it studies the differential in the theory as such a differential operator. Following this Malgrange \cite{malgrange1966cohomologie} and Johnson \cite{johnson1971some} discovered that the cohomology theory of Spencer can be translated into the language of sheaves of modules over the ring of differential operators $\D_X$. From their work we have the {\it algebraic Spencer complex}
\begin{equation}\label{algebraic-spencer}
    0 \to \D_X \otimes_{\O_X}\bigwedge^n \T_X \to \cdots\to \D_X \otimes_{\O_X}\bigwedge^2 \T_X \to \D_X \otimes_{\O_X}\T_X \to \D_X \to 0.
\end{equation}
This was used by Johnson to compute Spencer's cohomology groups in terms of $\Ext$ groups of carefully chosen $\D$-modules. The algebraic theory is adjoint to the topological theory in a precise sense investigated by Johnson ({\it ibid}).

Johnson notes that the theory wholly extends to the case of a smooth algebraic variety; however, the jump from the smooth case into the general case proved less fruitful than initially imagined as $\D$-modules have a number of undesirable properties on singular spaces \cite{bernvstein1972differential}. The algebraic theory has never properly been extended to singular spaces like general algebraic varieties. 

On the other hand it is immediate that the topological theory is definable in the algebraic category, taking $\T^*_X$ to be the sheaf of K\"ahler differentials of an algebraic variety $X$ and $\J^r_X(E)$ the jet sheaf of some sheaf $E$ on $X$ (note that this exists on a singular algebraic variety; see {\it e.g.} \cite[\S16]{ega-iv}, where no smoothness assumptions are used), and may or may not have good properties. This saw important usage in the work of Vinogradov in homological algebra concerning invariants produced by derivations of equations in modules, where he studied many cases of the topological complex on singular affine varieties. Based on this work, in \cite{vinogradov1979some} he conjectured that a more general version of the topological complex is acyclic on an affine variety if and only if the variety is smooth. In this sense the topological theory is still only expected to be a meaningful cohomology on smooth spaces, but the question has remained unresolved as such. 

In this paper we will build a theory of Spencer cohomology on singular schemes. The first result of this paper concerns this conjecture. We find a straightforward answer to the question posed by Vinogradov in the form of a class of counterexamples to the {\it only if} direction. In particular, a scheme being equipped with a certain endomorphism on its de Rham differentials is sufficient for acyclicity of the topological complex (Theorem \ref{only-if-thm}). In this sense we view the topological Spencer complex as a coarser theory than just detecting smoothness would allow. 

In later sections and for our main result, we seek to refine this measure by giving a suitable definition of an algebraic Spencer complex (and in turn Spencer cohomology) for any finite type scheme over a field of characteristic zero, with no assumptions on the types of singularities it may have other than imbeddability into a smooth scheme. The case of coherent $\D_X$-modules will prove simple as the formal completion is exact, meaning the homological properties of the $\D_X$-module are preserved. This is developed in \S\ref{spencer-complex-section} and culminates in Definition \ref{spencer-complex-def}. The case where we do not have coherent $\D_X$-modules demands more conceptual sophistication, but we see they are technically more pleasant. We will take the derived completion in this case. Some preliminaries will be required to establish what the properties of derived completions of $\D_X$-modules are \S\ref{theory-and-properties-sec} and after this interlude we will give a parallel of Definition \ref{spencer-complex-def} in the derived setting in \S\ref{derived-spencer-sec}. We will discuss what information about the space it contains and what makes it a suitable replacement of Spencer cohomology. Throughout we will find interesting connections to local cohomology obtained by the application of Kashiwara's equivalence, suggesting the completion is both (i) a natural viewpoint to adopt and (ii) a richer way to look at Spencer cohomology.

\section{Vinogradov's conjecture}\label{vinogradov-sec}

By $k$ we will denote a field of characteristic zero, but not necessarily algebraically closed. All rings will be assumed Noetherian unless stated otherwise. All schemes will be locally Noetherian and of finite type over $k$. We will call a (non-zero) Noetherian ring {\it Cohen--Macaulay} (CM) if, when regarded as a module over itself, every localisation at a prime ideal has equal dimension and depth. A scheme (as per our conventions) $X$ is CM if at every point the local ring $\O_{X,x}$ is CM. This and other definitions will be as according to the excellent textbook account \cite[\S 2]{bruns1998cohen}.

In \cite{vinogradov1979some} a version of \eqref{jet-complex-eq} is given for more general complexes of differential modules. Vinogradov's complex depends on an index $\sigma$ consisting of a (possibly infinite) sequence of positive integers, such that $\sigma=(s_1, s_2, s_3, \ldots)$ corresponds to adding differential operators of order $s_i$ in the $i$-th prolongation of $E$. For $\sigma = (1, \ldots, 1)$, $\abs{\sigma} = r$, his complex is the Spencer complex with jet sheaf of order $r$. Vinogradov conjectures that his complex is acyclic on an affine variety for all sequences $\sigma$ if and only if the variety is smooth. The {\it if} direction follows easily:

\begin{theorem}\label{if-thm}
    Let $X$ be a scheme (not necessarily affine) under the conventions above. If $X$ is smooth, then \eqref{jet-complex-eq} is acyclic for all $\sigma$.
\end{theorem}
\begin{proof}
    Since $X$ is smooth, $\T_X$ is locally free, so that at any stalk there is an \'etale-local set of coordinates in which $\partial_1, \ldots, \partial_n$ is a commuting basis. Now we will analyse \eqref{jet-complex-eq} at the stalks. According to \cite{johnson1971some} the differential of any jet sheaf-valued $(n-1)$-form $\omega$ in \eqref{jet-complex-eq} satisfies
    \begin{align*}
        (\dd{\omega})(\xi_1, \ldots, \xi_n) = &\sum_{j} (-1)^{j-1} \xi_j \cdot \omega(\xi_1, \ldots, \hat\xi_j, \ldots, \xi_n) \\&+ \sum_{1 \leqslant i < j \leqslant n} (-1)^{i+j} \omega([\xi_i, \xi_j], \xi_1, \ldots, \hat \xi_i, \ldots \hat \xi_j, \ldots, \xi_n).
    \end{align*}
    By linearity of the differential it is sufficient to evaluate it on a basis of $\T_X$. The commutator of basis elements vanishes; hence, from non-degeneracy, 
    \[
    \omega([\partial_i, \partial_j], \partial_1, \ldots, \hat \partial_i, \ldots, \hat \partial_j, \ldots, \partial_n)
    \]
    vanishes for all $i, j$. We are left with the first term of the differential. Since $\partial_1, \ldots, \partial_n$ generates $\T_{X,x}$, what remains is the differential of a Koszul complex of an isomorphism of free modules of rank $n$. By standard results such a complex is exact. The rest of the claim follows for all $\sigma$ by changing the value of $\omega$; namely, from differential forms valued in one sequence of prolongations to any other. 
\end{proof}

However, there are ultimately counterexamples in the singular case.

\begin{theorem}\label{only-if-thm}
    Let $X$ be scheme (not necessarily smooth) under the conventions above and $L_{\xi}$ the Lie derivative along a derivation $\xi$. If there exists a global derivation on $\Xi$ on $X$ such that for $i\geqslant 1$, the $\O_X$-linear endomorphism
    \[
        L_\Xi\colon \Omega_X^i \otimes_{\O_X} \J^r_X(E) \longrightarrow \Omega_X^i \otimes_{\O_X} \J^r_X(E)
    \]
    is an automorphism, then \eqref{jet-complex-eq} is exact. 
\end{theorem}
\begin{proof}
    One can verify with a brief computation that 
    \[
        L_\Xi =\dd{}\circ\iota_\Xi + \iota_\Xi \circ \dd{}
    \]
    on $\Omega_X^{\sbullet} \otimes_{\O_X} \J^r_X(E)$ with differential given above. For $i\geqslant 1$, the hypothesis makes $L_\Xi$ invertible, so that we may define
    \[
    h_i \coloneq \iota_\Xi \circ L_\Xi^{-1} \colon \Omega_X^i \otimes_{\O_X} \J^r_X(E) \to \Omega_X^i \otimes_{\O_X} \J^r_X(E).
    \]
    Now on $\Omega_X^i \otimes_{\O_X} \J^r_X(E)$ we compute
    \[
        \dd\circ h_i + h_{i+1} \circ \dd = \dd \circ \iota_\Xi \circ L_\Xi^{-1} + \iota_\Xi \circ L^{-1}_\Xi \circ \dd = L_\Xi L_\Xi^{-1}
    \]
    such that $h=\{h_i\}_{i\geqslant 1}$ is a contracting homotopy in positive degrees.
\end{proof}

Examples of affine schemes that satisfy these hypotheses include affine cones with an Euler vector field: if the $\mathbf{G}_m$-action gives a weight decomposition on each $\Omega_X^i$ with all weights nonzero for $i\geqslant 1$, then the Lie derivative along the Euler vector field acts diagonally with invertible eigenvalues on $\Omega_X^{\geqslant 1}$, so the hypothesis holds for the de Rham differentials. Taking sheaf-valued differential forms modifies nothing.

\begin{corollary}
    If $X$ is an affine scheme with such a vector field, then the complex of Vinogradov for $\sigma = (1, 1, 1, \ldots)$ is acyclic.
\end{corollary}

In fact we can make the same argument for any $\sigma$, by the concluding observation in the proof of Theorem \ref{if-thm}.



There is an adjoint version of \eqref{jet-complex-eq} due initially to \cite{malgrange1966cohomologie}. Let $\D_X$ be the ring of derivations on a smooth scheme $X$, and define the Spencer complex of $\D_X$ (following \cite{johnson1971some}; see also \cite[\S1.5]{hotta2007d}) as the resolution of $\O_X$
\[
    \Sp^{\sbullet}(\D_X) = \left[0 \to \D_X \otimes_{\O_X}\bigwedge^n\T_X \to \cdots \to\D_X \otimes_{\O_X}\bigwedge^2\T_X \to \D_X\otimes_{\O_X}\T_X \to \D_X \to 0\right]
\]
with differential in degree $k$
\begin{align*}
    \dd(m \otimes \xi_1 \wedge \ldots \wedge \xi_k) &= \sum_{j=1}^k (-1)^{j-1}(m\xi_j)\otimes \xi_1 \wedge \ldots \wedge \widehat{\xi}_j \wedge \ldots \wedge \xi_k \\ & + \sum_{1 \leqslant i < j \leqslant k} (-1)^{i+j} m\otimes [\xi_i, \xi_j] \wedge \xi_1 \wedge \ldots \widehat{\xi}_i \wedge \ldots \wedge \widehat{\xi}_j \wedge \ldots \wedge \xi_k.
\end{align*}

\begin{remark}\label{resolution-remark}
    The Spencer complex is a resolution of $\O_X$ (as a left $\D_X$-module) by locally free $\D_X$-modules because there is a $\D_X$-linear augmentation map $\D_X \to \O_X$ sending a differential operator $\xi$ to its evaluation at a constant $\xi(1)$; after that, the complex is locally a Koszul complex, from which exactness follows by the same argument as in Theorem \ref{if-thm}.
\end{remark}

This complex is adjoint in the sense that when $\T_X$ is locally free, one has
\[
    \Hom_{\D_X}\left(\D_X\otimes_{\O_X}\bigwedge^i \T_X, J^\infty(E)\right) = \bigwedge^i \T^*_X \otimes_{\O_X} J^\infty(E).
\]
with differential also dualised. 

Since the hom functor is not exact, the question for the algebraic Spencer complex is not implied by our results on the topological Spencer complex, making it interesting to pose Vinogradov's conjecture in that case. Moreover, the proof of Theorem \ref{only-if-thm} placed no assumptions on the class of singularity, but we expect the algebraic Spencer complex to be more sensitive to such data in the following sense. One can also consider any left $\D_X$-module $\M_X$ and obtain the complex
\[
    \Sp^{\sbullet}(\M_X) = \left[0 \to \M_X \otimes_{\O_X} \bigwedge^n\T_X \to \cdots \to \D_X \otimes_{\O_X} \bigwedge^2\T_X \to \D_X\otimes_{\O_X} \T_X \to \D_X \to 0\right]
\]
by properties of the tensor product. Let $X$ be a scheme of dimension $n$. There is an isomorphism between the Spencer complex of $\M_X$ and the $n$-shifted de Rham complex of $\M_X$, 
\[
    \Sp^{\sbullet}(\M_X) \simeq \Omega^{n+\sbullet}_X \otimes_{\O_X} \M_X
\]
which for $\M_X = \J^r_X(E)$ recovers \eqref{jet-complex-eq} on the right-hand side (up to a degree shift). Note that these are subtly different constructions---the de Rham cohomology of a $\D_X$-module arises from the pushforward of the $\D_X$-module to the point, `integrating out' the $n$-dimensional fibre directions (since this operation indexes the complex from $-n$ to zero, this accounts for the degree shift). That aside, we can see Spencer's original cohomology is equivalent to $\Sp^{\sbullet}(\J^r_X(E))$ up to said degree shift, in the sense that Johnson's isomorphism describes local resolutions of $\D$-modules on $X$ using the Spencer complex, whilst the shifted de Rham complex describes how the same cohomological information appears through global integration pairings extracting de Rham cohomology.

\begin{remark}
    Since the $\D_X$-module action term $m\xi_j$ in the differential is given by $\phi^rf\cdot\xi = \phi^r(\xi_j(f))$, we can interpret this cohomology as measuring obstructions to the integrability of $r$-jets by verifying compatibilities along a polyvector field ({\it i.e.}, agreement of differentiation in the directions prescribed by some exterior power of derivations). This is in the sense that the image of the differential is in the kernel of the term in degree less one precisely when different directions agree such that they cancel each other, meaning a jet is the derivative of a jet in degree less one when it respects compatibilities like Clairaut's theorem and is not merely a formal expression. Correspondingly, we recover Spencer's original motivation; namely, the study of the existence of solutions to partial differential equations, and its applications to notions of higher torsion invariants of manifolds (see {\it e.g.} \cite{guillemin1965integrability, guillemin1966deformation}).
\end{remark}

This is most obvious for the dualising sheaf $\M_X = \omega_X$. Using properties of smoothness of $X$, since there is a perfect pairing of polyvector fields on $\omega_X$ the interior product,
\[
\omega_X \otimes \bigwedge^i \T_X \xrightarrow{\sim} \Omega^{n-i}_X,
\]
$\Sp^{\sbullet}(\omega_X)$ reduces to the de Rham complex. This is now acyclic by the Poincar\'e lemma. Being CM implies various good properties from the viewpoint of algebraic topology; namely, it is the minimal condition for which intersection theory \cite{hochster1978some, fulton} and available duality theorems \cite{hartshorne1966residues, kovacs2013singularities} work `as expected'. This is evident in the algebraic complex. Working in the derived category, we can consider chain complexes of $\D_X$-modules. On a scheme with worse singularities than CM we do not have concentration of the dualising complex; instead we have the resolution 
\[
    \mathrm{Tot}\big(\omega^{\sbullet}_X \otimes \Sp(\D_X^{\sbullet})\big)
\]
which is not the de Rham complex in general. 

In general on a smooth scheme it is again true that locally the Spencer complex of a $\D_X$-module is a Koszul complex such that the same acyclicity argument applies. However, in contrast to \eqref{jet-complex-eq}, $\D_X$-modules (and hence their Spencer complexes) are poorly behaved on singular schemes---and in particular, there is in general no such pairing, owing to the failure of Poincar\'e duality. One concludes that the adjoint version of Spencer cohomology is more sensitive to the topology of the variety and may elucidate why CM is a sufficient property to have a cohomology theory with desirable properties like acyclicity and Poincar\'e duality. We will proceed by investigating what algebraic Spencer cohomology actually measures in the singular case and why it is not a resolution of the structure sheaf, and how a sensible definition can be created in replacement. 

\section{The na\"ive treatment of the singular case}\label{naive-singular-sec}

In this section we will ask what the pushforward of Spencer cohomology actually computes when we have a closed immersion of a singular scheme into a smooth one, and motivate a different construction with better homological properties. 

Let $X, Y$ be smooth $k$-schemes and consider a $\D_Y$-module $\M_Y$ and the sheaf of differential operators on $X$, $\D_X$. There is a natural pushforward operation of $\M_Y$ along $f\colon Y \to X$ by first pulling back the sheaf of differential operators on $X$, tensoring $\M$ with $f^*\D_X$ over $\D_Y$, and then taking the sheaf-theoretic direct image. There exists a derived counterpart to this operation, denoted by
\[
    f_+\M_Y \coloneqq {\bf R}f_*(\M_Y \otimes^{\bf L}_{\D_Y} f^*\D_X).
\]
By Kashiwara's equivalence the (derived) direct image is a $\D_X$-module with support on $Y$. 

Consider the pushforward of a $\D_X$-module $\M$ by $g\colon X \to \Spec k$, and note that $g^*\D_{\Spec k} = \O_X$, such that
\[
    g_+\M = {\bf R}g_*(\M \otimes^{\bf L}_{\D_X} \O_X).
\]
Applying the Spencer complex of $\D_X$ to resolve the derived tensor product (see Remark \ref{resolution-remark}), we have 
\[
    g_+\M \simeq {\bf R}g_*(\M \otimes_{\D_X} \Sp^{\sbullet}(\D_X)) =  {\bf R}g_*\Sp^{\sbullet}(\M)
\]
and hence the pushforward to the point computes the cohomology of $X$ with coefficients in $\M$. Let $f\colon Y \to X$ be a morphism of $k$-schemes with $X,Y$ smooth. From the projection formula there exists a natural isomorphism between the push forward by composite maps and the composition of pushforward maps. As such we can consider the resolution of the composite derived pushforward $(g\circ f)_+\M_Y$,
\begin{equation}\label{pushed-forward-complex-eq}
    \Sp^{\sbullet}(\M_X) \coloneqq \left[ \M_X \otimes_{\O_X}\bigwedge^n\T_X \to \cdots \to \M_X \otimes_{\O_X}\bigwedge^2\T_X \to \M_X \otimes_{\O_X}\T_X \to \M_X \right].
\end{equation}
One can show the composite pushforward computes the cohomology of $Y$. Let $h \colon Y \to \Spec k$ be the structure morphism of $Y$. Since both $X$ and $Y$ are $k$-schemes we have the fact that $g\circ f = h$, and we know that $h^*\D_{\Spec k} \simeq \O_Y$. Since $h_+\M_Y$ therefore computes the Spencer cohomology of $\M_Y$ on $Y$, we have the claim.

One says a quasi-coherent sheaf is {\it set-theoretically supported on a subscheme} if every section is annihilated by some power of the ideal of definition of the subscheme. Let $Y$ be singular. There are now a number of subtleties to the theory, but when $Y$ admits a closed imbedding into a smooth scheme $X$, one can use Kashiwara's equivalence to define a suitable category of $\D$-modules on a singular variety, whereby the pushforward is an equivalence between the category of $\D_Y$-modules on $Y$ and the subcategory of $\D_X$-modules set-theoretically supported on $Y$. Such a construction was perhaps first used in practice to handle $\D$-modules on singular subspaces by \cite{kashiwara1978holonomic}, and was known at least to Bernstein at around the same time.\footnote{This is mentioned, for instance, in his unpublished notes entitled ``Algebraic theory of $D$-modules''.} Note this is even true for imbeddings of cuspidal singularities into cuspidal $X$ by \cite{ben2004cusps}; moreover both agree with the intrinsic definitions of $\D_Y$-modules as crystals \cite{gaitsgory2011crystals} or as certain dg-algebras \cite{yang2022dmodules} (one can see this by using the fact that the semi-normalisation is smooth). In \cite{ben2004cusps}, certain universal homeomorphisms of CM varieties are used to control higher Ext groups and show the equivalence. 



In the general singular (imbeddable) case the pushforward becomes more delicate, and agrees with the previously defined direct image only up to an important homological point. Let $i\colon Y \to X$ be a closed imbedding with $Y$ singular and $\Mscr_Y$ denote the category of $\D_Y$-modules ($\Mscr_X$, $\D_X$-modules, respectively). There are functors
\[
    i^\natural\colon \Mscr_X \to \Mscr_Y, \quad i_\natural\colon \Mscr_Y \to \Mscr_X 
\]
satisfying the following property:
\begin{definition}
    There exists a $\D_X$-module supported on $i(Y)$ given by $\I_Y \D_X\backslash \D_X$ which is canonically identified with $\D_Y$ by taking its inverse image $f^\natural$; adjointly, the direct image $f_\natural$ of $\D_Y$ is identified with $\I_Y \D_X\backslash \D_X$ by Kashiwara's equivalence. This construction is independent of the imbedding and is local on affine neighbourhoods. The module $\D_Y$ has the property that for any $\D_Y$-module $\Hom(\D_Y, \M_Y)$ Similarly, we have 
    \[
    \M_Y = i^\natural(\I_Y \M_X \backslash \M_X) \quad i_\natural(\M_Y) = \I_Y \M_X \backslash \M_X 
    \]
    where $\M_X, \M_Y$ are objects of $\Mscr_X, \Mscr_Y$, respectively. 
\end{definition}

One can see \cite[\S A.2]{etingof2018poisson} for an account of the properties of this object.

Although this is a suitable definition of $\D_Y$, at this level of view, it lacks some desirable properties. In particular, with the direct image being now different, the argument above does not apply and it is no longer obvious we are resolving $\O_Y$. This is in the sense that set-theoretic support is equivalent to stalks being non-zero, so the stalks agree, meaning $i^\natural\T_Y$ is locally free on $i(Y)$ (as an $\O_Y$-module) if and only if $\T_Y$ was a locally free $\O_Y$-module in the first place. As such, the Koszul complex cannot be exact. 

In a precise sense the failure to resolve $\O_Y$ is because the Koszul complex detects homological information about the singularities through the irregularity of sections of $\T_Y$. Given the obvious augmentation map $\alpha\colon \T_Y \to \O_Y$ the degree zero homology of the Koszul complex is isomorphic to $\O_Y / \alpha(\T_Y)$ and the higher homology sheaves are supported on $V(\alpha(\T_Y))$. In many concrete situations these data coincide with known invariants controlling the local geometry around the singularity. For example, for a singular complex hypersurface in a smooth ambient variety over $\C$, the image of $\T_Y$ under evaluation is the Jacobian ideal, making $\O_Y/J$ the algebra whose dimension is the Milnor number of the singular fibre. Viewing the Spencer complex as a twist of the Chevalley--Eilenberg complex of the Lie algebra of derivations of $\O_Y$, more examples are available to us, following {\it e.g.} the results in \cite{elashvili2006lie}.

Moreover, the statement that sections are annihilated by some power of the ideal contains an ambiguity: in general there need not be one exponent which kills the entire sheaf. For this reason one might postulate an inverse limit over all powers of the ideal. In fact one knows already that the cohomology of a sheaf supported on a subspace uses the $I$-adic completion and was defined as such for exactly this reason (see {\it e.g.} the account in \cite{Hartshorne1967}). On the other hand, losing control on local cohomology obstructs us from certain good results such as \cite{ben2004cusps}. For this reason we claim the cohomology of such a $\D$-module should be thought of as a local cohomology due to the support condition, and that we should be taking the completion of the direct image in general. The aim of what is to follow will be to provide such a definition and consider its properties, especially with respect to singular spaces.

In addition to being a very natural construction, we will remark that this construction both uses the smoothness of $X$ to transfer certain nice properties to a $\D_Y$-module and the completion of $X$ along $Y$ to restore it to something with sole knowledge of $Y$. This closely parallels the construction of Hartshorne in \cite{hartshorne1975rham} and we will indeed see that for a certain $\D_Y$-module we recover his theory; in this sense we can view our theory as a generalisation of his construction to arbitrary cohomologies obtainable as the Spencer cohomology of some sheaf, simultaneously offering another viewpoint on why completion arises in his work. We will find that the smoothness transfers acyclicity properties to the completed complex such that it is a resolution of the formal completion of $\O_Y$, addressing the first obstruction to having a meaningful cohomology as well.

\section{Completed Spencer cohomology}\label{spencer-complex-section}

When it comes to constructing a complex on $Y$ which is acyclic, we have two competing {\it desiderata}: one is to compute the cohomology of $Y$, and the other is that $\D_Y$-modules are poorly behaved. The solution to the second is to use Kashiwara's equivalence and consider a $\D_X$-module with support on $Y$. However, when we take the pushforward of the dualising complex to $X$ and let it concentrate, if we restrict it to the singular subspace, it is no longer acyclic, and if we consider the entire sheaf, we are computing the cohomology of $X$. 

This is not surprising: on a singular scheme, the most basic example of de Rham cohomology $\Omega_Y^{\sbullet}$ is not a good resolution of $k_Y$, so we have no access to any cohomological properties of $Y$ using the na\"ive complex. Let $Y \hookrightarrow X$ be a closed embedding of schemes over a field of characteristic zero. One knows by the work of Grothendieck that when $Y$ is smooth, the hypercohomology of the de Rham complex is well-behaved \cite{grothendieck1966rham}, whereas when $Y$ is singular, we must consider the algebraic de Rham complex on the formal completion of $X$ along $Y$, in which case one can show it is independent of the choice of $X$ and enjoys various good properties \cite{hartshorne1975rham}.

Some motivation for why Hartshorne's argument works is as follows. In certain situations we can `restore' non-singularity by completion. One observes that singularities persist in the formal completion of a smooth scheme along a singular subscheme, since the formal neighbourhoods of the completion contain the ideal of definition of the singular subscheme. What is true instead is that the completion along a subscheme contains all of the infinitesimal normal data of the subscheme, making it like a formally smooth tubuluar neighbourhood; this can be thought of conceptually as modifying the sheaf of the space in a way that adds enough smoothness for the usual sheaf cohomology machinery one wants to use. However (and to reiterate) this space is still equivalent to the original subscheme in the sense that the underlying topological spaces are the same---making the cohomology meaningful. At a technical level the smooth ambient variety is simply a receptacle for a calculation on the crystalline site \cite{illusie1975report}---namely, by providing a scheme in which to find the tubular neighbourhood yielding our infinitesimal information.  In particular, the technical advantage is a formal Poincar\'e lemma providing the comparison theorem Hartshorne seeks, a form of which was known already to Grothendieck in the crystalline context \cite{grothendieck1968crystals}. 

In this section we will present a satisfactory definition of the Spencer cohomology of a singular scheme. We will opt to study the imbeddable case, leaving the generalisation to arbitrary schemes of finite type to a sequel. We will find that since $\O_X$ is coherent, the completion is not only exact but acts like a pullback to the underlying topological space of $Y$, making this cohomology meaningful.

\subsection{Smooth schemes and coherent sheaves}

Let $X$ be smooth (under the given conventions). We will first establish some auxiliary results for the case of coherent $\D_Y$-modules. 

Let $f\colon Y \to X$ be a closed imbedding of a singular variety into a smooth one. As discussed in \S\ref{naive-singular-sec}, {\it a priori} the direct image requires a different definition in the singular case. Namely, we have (in the notation of \cite{etingof2018poisson})
\[
    i_\natural\M_Y \simeq \I_Y \M_X\backslash \M_X, \quad \M_Y \simeq i^\natural(\I_Y \M_X\backslash \M_X)
\]
where $\M_X$ is a $\D_X$-module. These functors are mutually quasi-inverse equivalences between $\Mscr_Y$ and $\Mscr_X$. In the smooth setting, this is equivalent to the usual derived pushforward and pullback (one can see this by using smoothness to convert the quotient into a tensor product).

These functors are underived. There are the following counterparts on the bounded derived categories for arbitrary $f$: 
\[
    f_+\colon D^b(\D_Y) \to D^b(\D_X), \quad f^!\colon D^b(\D_X) \to D^b(\D_Y).
\]
Let $d = \dim Y - \dim X$. Explicit forms for these functors are available as 
\[
\R f_*\left( \M^{\sbullet}_Y \otimes^{\mathbf{L}}_{\D_Y} \D_{Y \to X}\right), \quad f^*\left(\M^{\sbullet}_X \otimes^{\mathbf{L}}_{f^*\D_Y} \D_{X \leftarrow Y}[d]\right)
\]
where $f_*, f^*$ are the direct and inverse images of sheaves and 
\[
\D_{Y \to X} = \O_Y \otimes_{f^* \O_X} f^*\D_X, \quad \D_{X \leftarrow Y} = \Omega_Y \otimes_{\O_Y} \D_{Y \to X} \otimes_{f^* \O_X} f^*\Omega^{-1}_X.
\]
For a closed imbedding $i$, $i_+$ coincides with $i_\natural$ ($i^!$, $i^\natural$, respectively) in the sense that $i_\natural \M^{\sbullet}_Y$ is the cohomology of $i_+\M^{\sbullet}_Y$ ($i^\natural$, $i^! \M^{\sbullet}_X$ where $\M^{\sbullet}_X$ has cohomology supported on $i(Y)$, respectively). Said another way, for a closed imbedding, $i_\natural$ is exact, and with the appropriate support condition on the image, $i^\natural$ is exact. 

\begin{remark}
    Note that some similar assertions as ours concerning $\D$-module operations extending uniquely to morphisms between formal schemes are made without complete proofs in \cite{beilinson1991quantization}. We choose to take the more arduous route of verifying the completion preserves these operations because of the ability to reuse many arguments in the derived setting to come. 
\end{remark}

In this section we will assume $\D_X$ is finitely-generated. This will be important in the proofs of Proposition \ref{completed-complex-prop} and Theorem \ref{completed-complex-thm} where it will guarantee that completion sends $\D_X$ to a locally free $\O_X$-module such that the completed Spencer complex is a complex of locally free $\D_X$-modules. Even when this assumption is unrealistic, we can take a finite stage of the filtration on $\D_X$, which will in general be coherent for a locally free $\D_X$-module. We will comment in various places about when statements for infinitely-generated $\D_X$ can be made and when looking at the filtration preserves those statements. 

\begin{definition}\label{completed-sheaf-def}
    Write $\widehat{X}$ for the completion of $X$ along $Y$. There exists a map of ringed spaces $\kappa\colon \widehat{X} \to X$ and the pullback of any $\O_X$-module $\mathcal{F}$ by $\kappa$, denoted
    \[
        \kappa^*\mathcal{F} = \mathcal{F}\otimes_{\O_X}\O_{\widehat{X}}.
    \]
    By \cite[10.8.8]{ega-i}, if $\mathcal{F}$ is coherent as an $\O_X$-module, then $\widehat{\mathcal{F}} \simeq \mathcal{F} \otimes_{\O_X} \O_{\widehat{X}}$. Thus we call the inverse image sheaf $\kappa^* \mathcal{F}$ the completion of $\mathcal{F}$. 
\end{definition}

Similarly to this definition, viewing $\O_{\widehat{X}}$ as a chain complex concentrated in degree zero, we have 
\begin{align}\label{completed-complex-eq}
    \begin{split}
        \qquad \qquad \qquad \kappa^*\mathcal{F}^{\sbullet} &= \mathcal{F}^{\sbullet}\otimes_{\O_X}\O_{\widehat{X}} \\ &= \mathcal{F}^n\otimes_{\O_X}\O_{\widehat{X}} \to \mathcal{F}^{n-1}\otimes_{\O_X}\O_{\widehat{X}} \to \ldots \to \mathcal{F}^0\otimes_{\O_X}\O_{\widehat{X}}
    \end{split}
\end{align}
with differential $\dd \otimes_{\O_X} \id_{\O_{\widehat{X}}}$. If $\mathcal{F}^{\sbullet}$ is coherent in each degree, then we have an isomorphism of chain complexes ({\it loc cit}).

\begin{remark}
The reader should mind that Definition \ref{completed-sheaf-def} involves an abuse of notation, as the expression for $\kappa^*\mathcal{F}$ should read 
\[
    \kappa^{-1}\mathcal{F}\otimes_{\kappa^{-1}\O_X}\O_{\widehat{X}}
\]
or equivalently---since the underlying topological space of $\widehat{X}$ is $Y$---
\[
    \mathcal{F}|_Y\otimes_{\O_X|_Y}\O_{\widehat{X}}.
\]
However, we can extend our restricted sheaves by zero, which does not change the formula. This applies identically to \eqref{completed-complex-eq}.
\end{remark}

On $X$ we may construct $\D_X$-modules which are coherent as $\O_X$-modules, being equivalent to vector bundles with integrable connections. This will be the setting of what follows. With reference to \eqref{completed-complex-eq}, we will consider
\[
    \kappa^*\Sp^{\sbullet}(\M_X) = \Sp^{\sbullet}(\M_X)\otimes_{\O_X}\O_{\widehat{X}}
\]
and the corresponding natural map of complexes $\kappa^*\Sp^{\sbullet}(\M_X) \xrightarrow{\sim} \widehat{\Sp^{\sbullet}(\M_X)}$. We will call the above product the completion of the Spencer complex of $\M_X$.

We will need the following lemma to generalise the statement that completion commutes with tensor operations to include the universal enveloping algebra. 

\begin{lemma}\label{completion-commutes-with-tensors-lemma}
    Let $F^pT(\F)$ be the $p$-th stage of the filtration of the tensor algebra of some coherent sheaf $\F$. The completion commutes with taking quotients of $F^pT(\F)$ for any finite $p$ so that $\widehat{F^pT(\F)} \simeq F^pT(\widehat{\F})$.
\end{lemma}
\begin{proof}
    The $p$-th stage of the filtration of the tensor algebra of a sheaf $\F$ is a length $p$ direct sum of multi-fold tensor products. Since completion commutes with finite direct sums it suffices to prove the claim for 
    \[
        \faktor{\F \otimes_{\O_X} \ldots \otimes_{\O_X} \F}{\mathcal I}.
    \]
    By the exactness of completion on coherent sheaves this is isomorphic to
    \[
        \faktor{\widehat{\F \otimes_{\O_X} \ldots \otimes_{\O_X} \F}}{\widehat{\mathcal{I}}}
    \]
    and by \cite[7.7.1]{ega-i} the tensor product commutes with completions to yield the claim.
\end{proof}

\begin{lemma}
    Let $X$ be smooth of dimension $m$ and $\widehat{X}$ be the completion of $X$ along some singular subscheme $Y = V(\mathcal I)$ of dimension $n$. The completion along $\mathcal I$ of the ring of differential operators of order at most $p$ on $X$ is isomorphic to the ring of differential operators of order at most $p$ on $\widehat{X}$.
\end{lemma}

\begin{proof}
    The sheaf of differential operators is equivalent to the universal enveloping algebra of the tangent sheaf, hence to a quotient of the tensor algebra of $\T_X$ by certain commutation relations. Take the Poincar\'e--Birkhoff--Witt filtration $F^p(\D_X)$ yielding differential operators of order at most $p$. This yields
    \[
        \faktor{F^pT(\T_X)}{F^p\mathcal{R}}
    \]
    where $T(\T_X)$ is the tensor algebra of $\T_X$ and $\mathcal{R}$ is the sheaf of ideals for the commutation relations. Now by Lemma \ref{completion-commutes-with-tensors-lemma} we have $F^p\widehat{U(\T_X)} = F^pU(\widehat{\T_X})$.
\end{proof}

As an abuse of notation we will write $\widehat{\D_X} \simeq \D_{\widehat{X}}$, though this is to be understood in the sense of a finite stage of the filtration. 

\begin{remark}
    Crucially, the stalk of a finite stage of the filtration of $\D_X$ is Noetherian. One can then verify as a sanity check that at the stalks we have an expression for the filtration of the ring of derivations on $\widehat{X}$, seen as a quotient of the free $\O_{\widehat{X}}$-algebra generated by derivations in coordinate directions of $Y$ and in infinitesimal normal directions inside $X$. By definition this is the stalk of $F^p(\D_{\widehat{X}})$. 
\end{remark}

\begin{corollary}
    For $\D_X$ finitely-generated, $\widehat{\D_X} \simeq \D_{\widehat{X}}$ on the nose. 
\end{corollary}

\begin{lemma}\label{tangent-sheaf-lem}
    Let $X$ be a smooth scheme locally of finite type and $\I \subset \O_X$ a coherent sheaf of ideals. One has an isomorphism $\widehat{\T_X} \simeq \T_{\widehat{X}}$.
\end{lemma}
\begin{proof}
    Observe that we have the isomorphism $\widehat{\T_X} = \widehat{\mathcal{H}om(\Omega^1_X, \mathcal{O}_X)} \simeq \mathcal{H}om(\widehat{\Omega^1_X}, \widehat{\O_X})$. Since $\Omega^1_X$ is coherent and completion commutes with taking differentials for finitely-presented modules, $\widehat{\Omega^1_X} \simeq \Omega^1_{\widehat{X}}$. Therefore $\widehat{\T_X} \simeq \mathcal{H}om(\Omega^1_{\widehat{X}}, \O_{\widehat{X}}) = \T_{\widehat{X}}$.
\end{proof}

\begin{lemma}\label{completed-d-mod-lem}
    If $\M$ is an $\O_X$-coherent $\D_X$-module then $\widehat{\M} \simeq \kappa^*\M$ as a $\D_{\widehat{X}}$-module.
\end{lemma}
\begin{proof}
    Take the completion
    \[
        \widehat{\M} = \varprojlim \faktor{\M}{\mathcal{I}^r\M}.
    \]
    By \cite[10.8.8]{ega-i}, if $\M$ is coherent as an $\O_X$-module, then $\widehat{\M} \simeq \kappa^{-1}\M \otimes_{\kappa^{-1}\O_X} \O_{\widehat{X}}$. As usual we can provide the $\O_X$-module inverse image the structure of a $\D_{\widehat{X}}$-module. Explicitly, choosing an $m \in \kappa^{-1}\M$ and denoting $\dd{\kappa} \colon \T_{\widehat{X}} \to \kappa^*\T_X$, from the Leibnitz rule we have the equation 
    \[
        \xi(s\otimes m) = \xi(s) \otimes m + s \cdot \dd{\kappa(\xi)(m)}
    \]
    for the action of derivations on $\widehat{\M}$.
\end{proof}

\begin{remark}
    In general Lemma \ref{completed-d-mod-lem} fails for $\M = \D_X$ because $\D_X$ is often not coherent as an $\O_X$-module, for instance, when we consider differential operators of arbitrary order. Another {\it caveat} is that this is not the action shown above is not in general the same action as the $\widehat{\D_X}$-action $\widehat{\M}$ inherits under completion. The assumption of finite-generation removes both obstructions.
\end{remark}

\begin{prop}\label{completed-complex-prop} 
\phantom{}
    \vskip0.5em
    \begin{thmnum}
    \setlength\itemsep{0.5em}
    \item The completion of the Spencer complex of $\D_X$, $\widehat{\Sp^{\sbullet}(\D_X)}$, is given by the complex whose $i$-th degree is 
    \[
        \widehat{\D_X} \,\,\widehat{\otimes}_{\O_X} \bigwedge^i \T_{\widehat{X}}.
    \] \label{completed-complex-prop-1}
    \item  Let $\D_X$ be finitely-generated. This complex is a complex of locally free $\widehat{\D_X}$-modules; identically, locally free $\D_{\widehat{X}}$-modules. \label{completed-complex-prop-2}
    \item The completion of the Spencer complex of $\M$ is the Spencer complex of the completion of $\M$. \label{completed-complex-prop-3}
    \item The completion of the Spencer complex of $\M$ is the completed tensor product of $\M$ and $\Sp^{\sbullet}(\D_X)$. \label{completed-complex-prop-4}
    \end{thmnum}
    \vskip0.5em
\end{prop}
\begin{proof}
    \phantom{}
    \vskip0.5em
    \begin{enumerate}
    \setlength\itemsep{0.5em}
    \item We will complete each degree separately. The claim follows by \cite[7.7.1]{ega-i}, the fact that completion commutes with tensor operations (such as exterior powers), and Lemma \ref{tangent-sheaf-lem}.
    \item  To see this, note that the Spencer complex consists of finitely-generated locally free $\D_X$-modules. Since completion commutes with finite direct sums, the completion is also locally free. There is an obvious action of $\widehat{\D_X}$ on each degree from completion. This promotes to an action of $\D_{\widehat{X}}$.
    \item This claim follows from the same argument that showed (1).
    \item Firstly we apply the same argument which we used to show (1). Now since $\widehat{\M}$ is naturally a $\widehat{\D_X}$-module, the tensor product
    \[
        \widehat{\M} \otimes \widehat{\D_X} \otimes \bigwedge^i \T_{\widehat{X}}
    \]
    absorbs the $\widehat{\D_X}$ factor into the action of derivations on $\widehat{\M}$. 
    \end{enumerate}
    \vskip0.5em
\end{proof}

\begin{remark}
    The assumption of $\D_X$ finitely-generated was only used in (ii) to establish local freeness. We can take a filtration of $\D_X$ and not lose any homological properties by the reasoning earlier in this section, culminating in Corollary \ref{pullback-cor} and Theorem \ref{completed-complex-thm}.
\end{remark}

\begin{definition}\label{spencer-complex-def}
    Let $Y$ be imbeddable as a closed subscheme into a smooth scheme $X$ lying over $k$ by a map of schemes $f$. Then we define the algebraic Spencer cohomology of a $\D_Y$-module $\M_Y$ by taking the hypercohomology of the completion along $Y$ of the pushed forward complex \eqref{pushed-forward-complex-eq},
    \[
        H^{\ast}_{\Sp}(Y,\M_Y) = \mathbf{H}^{\ast}\left(\widehat X,\widehat{\Sp^{\sbullet}(\M_X)}\right).
    \]
\end{definition}

Now we will investigate some properties of this cohomology. We will primarily be interested in decomposing the completed Spencer complex into completed pieces and producing the fact that the completed Spencer complex computes the cohomology of a formal scheme with coefficients in an $\mathcal I$-adically completed $\widehat{\D_X}$-module. 

Now in general we can say the following.

\begin{lemma}\label{formal-point-lemma}
    Let $\widehat{g}\colon \widehat{X} \to \mathrm{Spf}\,k$ be the map to the formal point induced by the completion of $g\colon X \to \Spec k$. The pullback $\widehat{g^*\D_{\Spec k}}$ is $\O_{\widehat{X}}$. 
\end{lemma}
\begin{proof}
    By definition we have $\widehat{g^*\D_{\Spec k}} = \widehat{\O_X \otimes_{g^{-1}\O_{\Spec k}}g^{-1}\D_{\Spec k}}$. Commutativity grants us $\widehat{\O_X \otimes_{g^{-1}\O_{\Spec k}}g^{-1}\D_{\Spec k}} = \O_{\widehat{X}}\widehat{\otimes}_{\widehat{g^{-1}\O_{\Spec k}}}\widehat{g^{-1}\D_{\Spec k}}$. Since $\Spec k$ is a one point space, $\D_{\Spec k} = \O_{\Spec k}$, from which the claim follows.
\end{proof}

\begin{corollary}\label{pullback-cor}
    The pullback to $X$ of a finite stage of the filtration on $\D_{\Spec k}$ is also $\O_{\widehat{X}}$ since all differential operators of order greater than zero must vanish.
\end{corollary}

From this we deduce that the completed Spencer cohomology arises naturally from pulling back a $\D_X$-module to $\widehat{X}$ and then pushing it forward to a formal point, so only requires the input data of Definition \ref{spencer-complex-def}.

\begin{prop}\label{formal-point-prop}
    Let $f\colon Y \to X$ and $\M_Y$ be as in Definition \ref{spencer-complex-def}, $f_+\M_Y$ be coherent as a sheaf of $\O_X$-modules, and define 
    \[
        \widehat{g}\colon\widehat{X} \to \Spf k.
    \]
    The cohomology $H^{\ast}_{\Sp}(Y, \M_Y)$ is equivalently 
    \[
        \widehat{g}_+(\kappa^*\M_X).
    \]
\end{prop}
\begin{proof}
    By \cite[10.8.8]{ega-i} we have level-wise isomorphisms of modules in each degree of the Spencer complex, and hence an isomorphism of chain complexes
    \[
        \kappa^*\Sp^{\sbullet}(\M_X) \xrightarrow{\sim} \widehat{\Sp^{\sbullet}(\M_X)}.
    \]
    As a consequence of Proposition \ref{completed-complex-prop}.(ii) we also have 
    \[
        \widehat{\Sp^{\sbullet}(\M_X)} \simeq \Sp^{\sbullet}(\widehat{\M_X}).
    \]
    Let $\widehat{g}\colon\widehat{X} \to \Spf k$. It follows that $H^\ast_{\Sp}(Y, \M_Y) = \widehat{g}_+\widehat{\M_X}$. From Lemma \ref{completed-d-mod-lem} we have $\widehat{\M_X} = i^*\M_X$. This proves the claim.
\end{proof}

\begin{theorem}\label{completed-complex-thm} 
\phantom{}
    \vskip0.5em
    \begin{thmnum}
    \setlength\itemsep{0.5em}
    \item The completed Spencer complex of $\widehat{\D_X}$ is a resolution of $\O_{\widehat{X}}$ by locally free $\widehat{\D_X}$-modules. \label{completed-complex-thm-1}
    \item  The completed Spencer complex of $\M$ computes the cohomology of $X$ with coefficients in $\widehat{\M}$. \label{completed-complex-thm-2}
    \end{thmnum}
    \vskip0.5em
\end{theorem}
\begin{proof}
    \phantom{}
    \vskip0.5em
    \begin{enumerate}
    \setlength\itemsep{0.5em}
    \item We will show that $\widehat{\Sp^{\sbullet}(\D_X)}$ resolves $\O_{\widehat{X}}$. We have established that this is isomorphic to $\Sp^{\sbullet}(\widehat{\D_X})$. Since the augmentation map 
    \[
        \D_X \to \O_X, \quad \xi \mapsto \xi(1)
    \]
    is locally a surjective $\O_X(U)$-module homomorphism, it is continuous in the $I$-adic topology (to see this recall that a homomorphism of modules $\varphi\colon M \to N$ has the property that $\varphi(I^nM)$ is contained in $I^n N$ for all $n$); following that observation, we can globalise the induced mapping to $\widehat{\D_X} \to \O_{\widehat{X}}$, sending a completed derivation $\widehat{\xi} \in \widehat{\D_X}$ to its evaluation at the image of the constant in $\O_{\widehat{X}}$, $\widehat{\xi}(\widehat{1}) \in \O_{\widehat{X}}$. This is also a surjection with kernel generated by completed differential operators with no zeroth-order part, {\it i.e.} $\T_{\widehat{X}}$. Let $(x_1, \ldots, x_n, y_{n+1}, \ldots, y_{n+m})$ be \'etale-local coordinates. To verify exactness, at a stalk we can apply the Koszul part of the differential to a basis for $\T_{\widehat{X}}$ given by $\partial_1, \ldots, \partial_{n+m}$. The claim follows.
    \item  We will complete the derived pushforward 
    \[
        g_+\M = {\bf R}g_*(\M \otimes^{\bf L}_{\D_X} g^*\D_{\Spec k}).
    \]
    It is useful to note that the derived completion commutes with derived pushforwards and derived tensor products, and that for coherent sheaves, completion is exact. For this reason we have
    \[
        \widehat{{\bf R}g_*(\M \otimes^{\bf L}_{\D_X} g^*\D_{\Spec k})} \simeq {\bf R}g_*(\widehat{\M} \widehat{\otimes}^{\bf L}_{\D_X} \widehat{g^*\D_{\Spec k}})
    \]
    In turn, by Lemma \ref{formal-point-lemma} this is equivalent to
    \[
        {\bf R}g_*(\widehat{\M}\,\, \widehat{\otimes}^{\bf L}_{\D_X} \O_{\widehat{X}}).
    \]
    Now we apply the first part of the claim to obtain
    \[
    {\bf R}g_*\left(\widehat{\M}\,\, \widehat{\otimes}_{\D_X} \Sp^{\sbullet}(\widehat{\D_X})\right).
    \]
    \end{enumerate}
    \vskip0.5em
\end{proof}

More will be said about the derived completion in the following section. 

It is possible to show that this definition depends only on the formal neighbourhoods of $Y$ and not on the imbedding, by an application of Kashiwara's equivalence: 

\begin{theorem}\label{indep-thm}
    The algebraic Spencer cohomology of the pair $(Y, \M_Y)$ is determined by the formal neighbourhoods of $Y$, and in particular, is independent of the choice of ambient space $X$ into which $Y$ is imbedded.
\end{theorem}
\begin{proof}
    Let $\mathbf{Coh}(\O_X)$ be the category of coherent sheaves of $\O_X$-modules and $\mathbf{Coh}(\O_X, \I)$ be the category of inverse systems of coherent sheaves of $\O_X$-modules, with $\I$ the quasi-coherent sheaf of ideals of $Y$. There is an obvious functor 
    \[
        \widehat{\phantom{...}}\colon \Coh(\O_X) \to \Coh(\O_X, \I).
    \]
    Completion is exact on coherent sheaves, so there is moreover an induced functor on bounded derived categories
    \[
        \widehat{\phantom{...}}\colon D^b_{\coh}(\O_X) \to D^b_{\coh}(\O_X, \I).
    \]
    Let $g\colon X \to \Spec k$, and denote by $\Mscr_X$ the category of $\D_X$-modules seen as bounded chain complexes concentrated in degree zero. By functoriality of $g_+$ we also have a functor
    \[
        \Sp^{\sbullet}_X \colon \Mscr_X \to D^b(\D_X).
    \]
    Restricting to the subcategory of $\D_X$-modules which are coherent as $\O_X$-modules, and its derived category $D^b_{\coh}(\D_X)$, we can compose these two functors:
    \[
        \widehat{\Sp^{\sbullet}_X}\colon \Coh(\D_X) \to D^b_{\coh}(\D_X) \to D^b_{\coh}(\O_X, \I).
    \]
    By Lemma \ref{completed-d-mod-lem}, completion preserves this subcategory, so that at last we have 
    \[
        \widehat{\Sp^{\sbullet}_X}\colon \Coh(\D_X) \to D^b_{\coh}(\D_X) \to D^b_{\coh}(\D_X, \I).
    \]
    Now take two schemes over $k$, denoted $V_1, V_2$, and consider a pair of morphisms $Y \to V_1$ and $Y \to V_2$. By Kashiwara's equivalence the subcategory of $\D_{V_1}$-modules supported on $Y$ is equivalent to the category of $\D_Y$-modules, and likewise for the subcategory of $\D_{V_2}$-modules supported on $Y$. It follows from transitivity that $\Coh(\D_{V_1}|_Y)$ and $\Coh(\D_{V_2}|_Y)$ are equivalent. Consider the diagram
    \[
    \begin{tikzcd}
        \Coh(\D_{V_1}|_Y) \rar["\sim", shift left] \ar[d, "\widehat{\Sp_{V_1}^{\sbullet}}", swap] & \lar[shift left] \Coh(\D_{V_2}|_Y) \ar[d, "\widehat{\Sp_{V_2}^{\sbullet}}"]\\
        D^b_{\coh}(\D_{V_1}|_Y, \I) \ar[r] & D^b_{\coh}(\D_{V_2}|_Y, \I) 
    \end{tikzcd}
    \]
    It is an exercise to verify that if $F\colon A \to B$ is an equivalence of abelian categories then there is an induced equivalence of derived categories $\tilde{F} \colon D^b_{\coh}(A) \to D^b_{\coh}(B)$. Essential surjectivity gives us object-wise isomorphisms in the derived category, and hence, isomorphisms on cohomology. 
\end{proof}

\begin{remark}
    In \S\ref{naive-singular-sec} it was remarked that we can imbed into a CM variety and still have desirable results. This remains true in the completion, since the completion of a CM ring is again CM.  
\end{remark}

\subsection{Revisiting the de Rham cohomology}

A consequence of Proposition \ref{formal-point-prop} is that $H^{\ast}_{\Sp}(Y, \omega_Y)$ recovers the completed de Rham cohomology of Hartshorne. As an interesting example, we will spell out the construction here. 

First we will show the pushforward takes dualising sheaves to dualising sheaves. In particular, let $f\colon Y \to X$ be a closed immersion. We will show that if $\omega_Y^{\sbullet}$ is a dualising complex on $Y$, then $f_* \omega_Y^{\sbullet}$ is a dualising complex for the coherent $\O_X$-algebra $f_*\O_Y$. By uniqueness of dualising complexes, this determines the dualising complex on all of $X$. 

Firstly, one has
\[
f_* \omega_Y^{\sbullet} \simeq \mathbf{R}\mathcal{H}om_{\O_X}(f_* \O_Y, \omega_X^{\sbullet}).
\]
To prove this we proceed as follows. We note that $f$ is a finite proper morphism because it is a closed immersion. Let $\overline{f}\colon (Y,\O_Y)\to (X,f_*\O_Y)$ to be the induced morphism of locally ringed spaces and let $f^\flat\colon D^b(X) \to D^b(Y)$ be defined as
\[
f^\flat = \overline{f}^* \mathbf{R} \mathcal{H}om_{\O_X}(f_*\O_Y,-).
\] Then according to \cite[\S III.6.5]{hartshorne1966residues} we have that the map
\[
\tau\colon \mathcal{H}om_{\O_X}(f_*\O_Y,G) \to \overline{f}_*\overline{f}^*\mathcal{H}om_{\O_X}(f_*\O_Y,G)
\]
induces an isomorphism
\[
\mathbf{R}\tau\colon \mathbf{R}\mathcal{H}om_{\O_X}(f_*\O_Y,G^{\sbullet})\to \mathbf{R}f_*f^\flat(G^{\sbullet})
\]
for $G^{\sbullet} \in D^+_{\mathrm{qc}}(X)$.
Now by \cite[\S V.2.4]{hartshorne1966residues} we know that $f^\flat\omega_X^{\sbullet} = \omega_Y^{\sbullet}$. Hence we get an isomorphism
\[
\mathbf{R}f_*\omega_Y^{\sbullet} \simeq \mathbf{R}\mathcal{H}om_{\O_X}(f_*\O_Y,\omega_X^{\sbullet})
\]

We will note some other facts.

(i) Since $f$ is proper, $f_*\O_Y$ is finite over $\O_X$, so that a complex of finite injective dimension over $\O_X$ has finite injective dimension over $f_*\O_Y$. 

(ii) If $\omega_X^{\sbullet}$ is dualising for $\O_X$ then for any coherent $F$ one has 
\[
    \mathbf{R}\mathcal{H}om_{\O_X}(\mathbf{R}\mathcal{H}om_{\O_X}(F, \omega_X^{\sbullet}), \omega_X^{\sbullet}) \simeq F.
\]

(iii) If $A$ is a coherent $\O_X$-algebra and $\omega_X^{\sbullet}$ is dualising for $\O_X$, then for any coherent $A$-module $M$, it holds that 
\[
    \mathbf{R}\mathcal{H}om_{\O_X}(M, \omega_X^{\sbullet}) \simeq \mathbf{R}\mathcal{H}om_A(M, \mathbf{R}\mathcal{H}om_{\O_X}(A, \omega_X^{\sbullet})).
\]
Now we can prove the following theorem.

\begin{theorem}
Let $f: Y \to X$ be a proper morphism of Noetherian schemes with $X$ admitting a dualising complex $\omega_X^{\sbullet}$. If the natural isomorphism $\mathbf{R}f_*\omega_Y^{\sbullet} \simeq \mathbf{R}\mathcal{H}om_{\O_X}(f_*\O_Y, \omega_X^{\sbullet})$ holds, then $f_*\omega_Y^{\sbullet}$ is a dualising complex for the coherent $\O_X$-algebra $f_*\O_Y$.
\end{theorem}
\begin{proof}
We will verify the three conditions for $f_*\omega_Y^{\sbullet}$ to be dualising for $f_*\O_Y$ \cite[\S V.2]{hartshorne1966residues}. 

Firstly, since the derived hom preserves injective resolutions, the complex 
\[
    \mathbf{R}\mathcal{H}om_{\O_X}(f_*\O_Y, \omega_X^{\sbullet})
\]
has finite injective dimension over $\O_X$. Now (i) implies it has the same property over the finite extension $f_*\O_Y$. 

Secondly, the reflexivity isomorphism $f_*\O_Y \to \mathbf{R}\mathcal{H}om_{f_*\O_Y}(f_*\omega_Y^{\sbullet}, f_*\omega_Y^{\sbullet})$ follows from a brief computation. Using the isomorphism $\R\tau$ and tensor-hom adjunction, we have 
\begin{align*}
\mathbf{R}\mathcal{H}om_{f_*\O_Y}(f_*\omega_Y^{\sbullet}, f_*\omega_Y^{\sbullet}) &\simeq \mathbf{R}\mathcal{H}om_{f_*\O_Y}(\mathbf{R}\mathcal{H}om_{\O_X}(f_*\O_Y, \omega_X^{\sbullet}), \mathbf{R}\mathcal{H}om_{\O_X}(f_*\O_Y, \omega_X^{\sbullet})) \\ &\simeq \mathbf{R}\mathcal{H}om_{\O_X}(\mathbf{R}\mathcal{H}om_{\O_X}(f_*\O_Y, \omega_X^{\sbullet}) \otimes^{\mathbf{L}}_{f_*\O_Y} f_*\O_Y, \omega_X^{\sbullet}) \\ &\simeq \mathbf{R}\mathcal{H}om_{\O_X}(\mathbf{R}\mathcal{H}om_{\O_X}(f_*\O_Y, \omega_X^{\sbullet}), \omega_X^{\sbullet}).
\end{align*}
The last expression is isomorphic to $f_*\O_Y$ by (ii); the claim follows. 

Thirdly, for any coherent $f_*\O_Y$-module $M$, we have 
\[
\mathbf{R}\mathcal{H}om_{f_*\O_Y}(M, f_*\omega_Y^{\sbullet}) \simeq \mathbf{R}\mathcal{H}om_{f_*\O_Y}(M, \mathbf{R}\mathcal{H}om_{\O_X}(f_*\O_Y, \omega_X^{\sbullet})).
\]
From (iii) the last expression is isomorphic to $\mathbf{R}\mathcal{H}om_{\O_X}(M, \omega_X^{\sbullet})$. Now since $M$ is coherent over $f_*\O_Y$, hence over $\O_X$, and $\omega_X^{\sbullet}$ is dualising for $\O_X$, this complex has coherent cohomology over $\O_X$, hence over $f_*\O_Y$.
\end{proof}

From this theorem it follows that one can reconstruct duality on $Y$ from duality on $f_*\O_Y$-modules. Now by uniqueness of $\omega_X^{\sbullet}$ and the fact that we have an adjunction between $f_*$ and $f^!$, the extension $\omega_X^{\sbullet}$ such that $f^\flat\omega_X^{\sbullet} \simeq \omega_Y^{\sbullet}$ is uniquely determined by agreement with the complex on $Y \subset X$.

Now consider 
\[
\Sp(\omega_X) = \left[\omega_X \otimes_{\O_X}\bigwedge^n\T_X \to \cdots \to\omega_X \otimes_{\O_X}\bigwedge^2\T_X \to \omega_X  \otimes_{\O_X}\T_X \to \omega_X  \right].
\]
By smoothness, on any $X$ we have 
\[
\Sp(\omega_X) = \Big[\Omega^0_X \to \cdots \to\Omega^{n-2}_X \to \Omega_X^{n-1} \to \Omega^n_X  \Big].
\]
Now take the pullback by $i$ of the sheaf in each degree and apply Proposition \ref{completed-d-mod-lem} to define the differential. 

Since each $\Omega^k_Y$ is a $\D_Y$-module, it follows by Theorem \ref{indep-thm} that the complex is independent of $X$. That is to say, putting $\omega_Y$ in Theorem \ref{indep-thm}, one readily sees that the result \cite[Theorem 1.4]{hartshorne1975rham} follows.

\begin{corollary}
    Completed de Rham cohomology is independent of the ambient scheme and depends only on $Y$.
\end{corollary}

This is, by \cite{beilinson1991quantization}, a crystalline theory, which we hope to investigate more in the future. The notion of a $\D$-module on the formal completion along a singular subvariety is evidently closely related to that of a $\D$-crystal, and in fact by classical results they are equivalent \cite[\S7.10]{beilinson1991quantization}. From this viewpoint one can see Hartshorne's theorems\footnote{In fact the relation between crystalline, and completed de Rham, cohomology in the singular case was first observed in unpublished lectures of Deligne (1969), with the related paper \cite{lieberman1971duality} following that.} as giving a satisfactory resolution to a conjecture of Grothendieck \cite[\S4]{grothendieck1968crystals}. This is in the sense that he shows sheaf cohomology on the crystalline site, whose open sets (in characteristic zero) are nilpotent thickenings of Zariski open subschemes, admits a quasi-isomorphism to the hypercohomology of the completed de Rham algebra, thus representing it in the derived category.

In the above we computed the pairing and completed to get the completed differentials. We ought to also be able to completed both arguments and get a pairing on completed modules which gives us the completed differentials. As a sanity check we verify this here. 

\begin{theorem}
    Let $X$ be of dimension $n$. Completing $\omega_X \otimes \wedge^i \T_X$ commutes with computing the perfect pairing on $\widehat{\omega_X}$.
\end{theorem}
\begin{proof}
    By the arguments in \cite{lipman} we have that $i^*\omega_X$ is dualising on $\widehat{X}$. We will denote this sheaf by writing $\omega_{\widehat{X}} \coloneqq i^*\omega_X$. Since $X$ is smooth, $\omega_X\simeq\bigwedge^{n}\Omega_{X/k}$. Now note that pulling back commutes with tensor operations such that 
    \[
        i^*\omega_X = i^*\left(\bigwedge^{n}\Omega_{X/k}\right) = \bigwedge^{n}i^*\Omega_{X/k}.
    \]
    Since $\Omega_{X/k}$ is coherent, by \cite[10.8.8]{ega-i}, we have the isomorphism 
    \[
        \bigwedge^{n}i^*\Omega_{X/k} \simeq \bigwedge^{n}\widehat{\Omega_{X/k}}
    \]
    ultimately implying $\omega_{\widehat{X}} = \widehat{\Omega^n_{X}}$. Since pullbacks also preserve local freeness and $X$ is smooth, $\widehat{\Omega_{X/k}}$ is locally free. By this fact we have a perfect pairing 
    \[
        \omega_{\widehat{X}} \widehatotimes \bigwedge^i \T_{\widehat{X}} \xrightarrow{\sim} \widehat{\Omega^{n-i}_{X/k}}
    \]
    obtained through contraction of completed $n$-forms with completed polyvector fields. By commutation of completion with tensor operations this is the same as $\Omega^{n-1}_{\widehat{X}}$.
\end{proof}

\begin{corollary}
    From the properties of the pairing there exists an isomorphism 
    \[
        \Sp^{\sbullet}(\M_Y) = \Omega^{n+\sbullet}_{\widehat{X}} \,\,\widehat{\otimes} \,\,\widehat{f_+\M_Y}.
    \]
    It then follows that the completed Spencer complex is acyclic by the results in \cite{hartshorne1975rham}.
\end{corollary}

\section{The theory and properties of derived complete \texorpdfstring{$\D$}{D}-modules}\label{theory-and-properties-sec}

We will now study the derived completion of a $\D$-module. Since $\D$-modules are naturally objects of the derived category, especially on singular varieties, and the homology of complexes of $\D$-modules offers important information, we will study the derived completion. In particular, completion is only exact on coherent sheaves---a restrictive assumption in the singular case. 

Let us introduce some of the basic ideas. Recall that the $I$-adic completion of a module is 
\[
    \widehat{M} = \varprojlim_r \faktor{M}{I^r M}
\]
and a module $M$ is $I$-adically complete when the natural map $M \to \widehat{M}$ is an isomorphism. In this sense we distinguish between {\it completed} and {\it complete} objects. If $I$ is finitely-generated, a completed module is complete. 

We can apply this to a chain complex by working in degrees. In general a chain complex is called completed if it is a complex of $I$-adically completed $\O_X$-modules---namely, for a chain complex of $\O_X$-modules $(C^{\sbullet}, \dd)$, we have the $I$-adic completion $\widehat{C^{\sbullet}}$ obtained by completing in degrees:
\[
    \widehat{C^k} = \varprojlim_{r} \faktor{C^k}{I^rC^k}.
\]
By the power rule the image of the differential of $I^r$ is contained in $I^{r-1}$, so that the differential is compatible with the inverse system and one can define
\[
    \widehat{\dd} \colon \widehat{C^k} \to \widehat{C^{k-1}}.
\]
A complex is $I$-adically complete if it is a complex of $I$-adically complete modules. There are two further distinctions at the level of chain complexes detailed in \cite{pol2022homotopy}. We will mainly concern ourselves with the notion of derived completeness. The derived completion of a sheaf of $\O_X$-modules is defined as follows. Let $K^{\sbullet}$ be an object of the (unbounded) derived category of $\O_X$-modules $D(\O_X)$ and $\I \subset \O_X$ a sheaf of ideals. Take an open neighbourhood of $U$ of $X$ and consider the chain complex of modules $M^{\sbullet} \coloneqq K^{\sbullet}|_U \in \mathrm{Obj}(D(\O_X|_U))$ and $I \coloneqq \I(U)$. We say $K^{\sbullet}$ is $\I$-adically derived complete if for all $U$, for every $f \in I$ we have 
\[
    \R\varprojlim_{r}\left(\ldots M^{\sbullet} \xrightarrow{f} M^{\sbullet}  \xrightarrow{f} M^{\sbullet}\right) = 0.
\]
Here the inverse system should be interpreted as a diagram of chain complexes with chain mappings from multiplication by some element of the ideal, and the derived limit as a homotopy limit. By construction every derived completion is derived complete.

We will note that for a Noetherian ring $A$ and a finitely-generated $A$-module, the $I$-adic completion is exact; having level-wise isomorphisms further means that the $I$-adic completion of a chain complex is equivalent to the derived $I$-adic completion of that chain complex. Similarly, completion is exact on coherent sheaves. However, since our ultimate goal is to investigate singular $\D$-modules where often $\M$ is not coherent, we will require the more general machinery. In fact, we have already seen that even in the smooth case, because of the derived nature of $\D$-modules, the derived completion is convenient (Theorem \ref{completed-complex-thm}).

We will denote by $D_{\mathrm{comp}}(\O_X)$ the category of derived complete $\O_X$-modules. There is an obvious inclusion of $D_{\mathrm{comp}}(\O_X)$ into $D(\O_X)$ and by general theory this inclusion has a left adjoint we will call derived completion. An expression for this is possible by noting that, whenever certain higher Tor functors vanish, $M / I^r M \simeq M \otimes_A A/I^r$, such that
\[
    \widehat{M} = \varprojlim_r \left(M \otimes_A \faktor{A}{I^r}\right).
\]
A desirable substitute for this construction in the derived category is given by resolving the quotient with the Koszul complex, yielding
\begin{equation}\label{derived-completion-eq}
    \widehat{M^{\sbullet}} = \R\varprojlim_r \big( M^{\sbullet} \otimes^{\mathbf{L}}_A \mathrm{Kos}(A; f_1^r, \ldots, f_\tau^r)\big)
\end{equation}
with $f_1, \ldots, f_\tau$ being generators of $I$.

The derived completion carries higher homological information correcting the failure for the underived completion of an object to be exact, and in particular, detects torsion and limit phenomena the underived completion does not see. Taking the cohomology of the underived completion therefore forgets higher obstructions and is generally not the correct set of invariants to take of a completed sheaf, whereas the derived completion contains the corrected invariants of a completed sheaf. In fact taking
\[
    H^\ast(X, \widehat{\F}^{\sbullet}) = \mathbf{H}(\R\Gamma(X, \widehat{\F}^{\sbullet}))
\]
we can see that
\[
    \R\Gamma(X, \widehat{\F}^{\sbullet}) = \widehat{\R\Gamma(X, \F^{\sbullet})}
\]
and hence by a certain spectral sequence argument 
\[
    H^\ast(X, \widehat{\F}^{\sbullet}) = \mathbf{H}(\widehat{\R\Gamma(X, \F^{\sbullet})}),
\]
so that the cohomology of the derived completion is the correct completion of the cohomology. Another useful property is that for a finitely-generated ideal, taking the derived completion along $\I$ is right adjoint {\it via} the tensor-hom adjunction to taking the derived local cohomology with respect to $\I$ (and {\it vice versa} with derived local cohomology the left adjoint). This result is due to \cite{tarrio1997local} where they provide a sheafification of the celebrated Greenlees--May duality \cite{greenlees1992derived}.

\section{Derived completed Spencer cohomology}\label{derived-spencer-sec}

We will now assume $\D^{\sbullet}_X$ and all $\D^{\sbullet}_X$-modules $\M^{\sbullet}$ are objects of the (unbounded) derived category of sheaves, and that $Y$ is singular and admits an imbedding into a smooth scheme (according to our conventions, of finite type) $X$. Then we will study the derived completion of the Spencer complex. Some digressions will be required about the behaviour of $\D$-modules and the Spencer complex in the derived complete category in order to define the derived completion of the Spencer complex and discuss its relation to the Spencer complex of the derived completion of the ring of differential operators $\Sp^{\sbullet}(\widehat{\D_X})$. From there it will be simple to reproduce the derived completed Spencer complex by resolving a certain derived completed derived tensor product. This will require reasoning about the exactness of the complex, linking our definition to the motivation from \S\ref{vinogradov-sec}. 

Just as in \S3, we can define the bounded derived category of quasi-coherent $\D_Y$-modules as the full subcategory of the bounded derived category quasi-coherent $\D_X$-modules consisting of complexes whose cohomology sheaves are supported on $Y$. Note that a more general version of Kashiwara's equivalence is proven in \cite[\S VI.7]{borel} without using exactness of the closed imbedding. Note also that derived completion with respect to an ideal of depth $d$ sends $D^{[a, b]}(\O_X)$ to $D^{[a - d, b]}(\O_X)$, meaning for a finitely-generated ideal, the adjunction described above descends to the bounded derived category, as do $i_+$ and $i^!$.

Our main claim will be that the derived completion of the derived tensor product is a suitable definition of Spencer cohomology in the derived category in the sense that its cohomology sheaves are completed Spencer cohomology. 

We will see that it is somewhat more straightforward to discuss things in the derived context. In general to compute something like \eqref{derived-completion-eq} we will want to also resolve $M^{\sbullet}$ and take the derived limit of the resulting total complex. To begin let $X$ be smooth and $\M^{\sbullet}_X$ a bounded chain complex of $\D_X$-modules. 

\begin{prop}\label{quasi-iso-prop}
    The total complex $\mathrm{Tot}(\M^{\sbullet}_X \otimes \Sp^{\sbullet}(\D_X))$ is a resolution of $\M^{\sbullet}_X$.
\end{prop}
\begin{proof}
    Suppose that for each $i$, we have a quasi-isomorphism
    \[
    F^{i,{\sbullet}} \longrightarrow M^i
    \]
    where $F^{i,{\sbullet}}$ is a bounded complex of flat $\mathcal{O}_X$-modules. Then $F^{{\sbullet}, {\sbullet}}$ forms a double complex with $i$ indexing columns. We claim the total complex is quasi-isomorphic to $M^{\sbullet}$. To show this we can use the standard spectral sequence argument for double complexes. Filter $\mathrm{Tot}(F^{\bullet,\bullet})$ by columns (the $i$-index). The $E_1$ page is
    \[
        E_1^{p,q} = H^q(F^{p,\bullet}) \simeq 
        \begin{cases} 
            M^p, & q=0,\\
            0, & q\neq 0,
        \end{cases}
    \]
    by the assumption that $F^{p,\bullet} \to M^p$ is a quasi-isomorphism. We know the spectral sequence converges to the cohomology of the total complex:
        \[
        E_1^{p,q} \implies H^{p+q}(\mathrm{Tot}(F^{\bullet,\bullet})).
        \]
    Since $E_1^{p,q}=0$ for $q\neq 0$, the spectral sequence collapses at $E_1$ and
    \[
        H^i(\mathrm{Tot}(F^{{\sbullet}, {\sbullet}})) \simeq H^i(M^{\sbullet}).
    \]
    The natural map
    \[
    \mathrm{Tot}(F^{{\sbullet}, {\sbullet}}) \longrightarrow M^{\sbullet}
    \]
    induces the identity on $H^i$ {\it via} the augmentation maps $F^{p,\bullet}\to M^p$, so it is a quasi-isomorphism between the total complex and $M^{\sbullet}$. Replacing $M^{\sbullet}$ with $\M_X^{\sbullet}$ and taking the total complex of its Spencer resolution, the statement is proven. 
\end{proof}

\begin{definition}\label{derived-complete-tensor-def}
    Let $K^{\sbullet} \mapsto \widehat{K^{\sbullet}}$ be the left adjoint \eqref{derived-completion-eq}. The derived completion of the derived tensor product 
    \[
        \widehat{K^{\sbullet} \otimes^{\mathbf{L}} L^{\sbullet}}
    \]
    satisfies a tensor-hom adjunction in $D_{\mathrm{comp}}(\O_X)$ such that it is identified as the tensor product in $D_{\mathrm{comp}}(\O_X)$. We will denote this as $- \,\,\widehat{\otimes}^{\mathbf{L}} - \colon D_{\mathrm{comp}}(\O_X) \times D_{\mathrm{comp}}(\O_X) \to D_{\mathrm{comp}}(\O_X)$. In particular, we have 
    \[
        \widehat{K^{\sbullet} \otimes^{\mathbf{L}} L^{\sbullet}} = \widehat{K^{\sbullet}}\,\,\widehat{\otimes}^{\mathbf{L}} \widehat{L^{\sbullet}}.
    \]
\end{definition}

\begin{lemma}
    Let $\F$ be a sheaf of $\O_X$-modules, $\I$ the ideal of definition of some subsheaf generated by sections $f_1, \ldots, f_n$, and $\widehat{\F}_{\I}$ the $\I$-adic completion of $\F$. Let $\mathcal{J}$ be some other ideal in $\O_X$ generated by sections $h_1, \ldots, h_n$. There exists a Koszul complex 
    \[
        \widehat{\mathrm{Kos}(\F; h_1, \ldots, h_n)_{\I}}
    \]
    which is exact when a regular sequence can be found in $\widehat{\F}$ whose intersection with $\I$ is empty.
\end{lemma}
\begin{proof}
    The Koszul complex of $s\colon \F \to \O_X$ is the chain complex 
    \[
        0 \to \bigwedge^n \F \to \ldots \to \bigwedge^1 \F \to \O_X \to 0
    \]
    with differential $\dd = \sum (-1)^{i+1} s(h_i) h_1 \wedge \ldots \wedge \hat{h_i} \wedge \ldots \wedge h_n$
    whose homology in degree zero is 
    \[
        H^0\left(\mathrm{Kos}(\F; h_1, \ldots, h_n)\right) \simeq \faktor{\F}{\mathcal{J}\F} = \faktor{\O_X}{\mathcal{J}} \otimes \F.
    \]
    From the disjointness of the ideal the derived completion of the Koszul complex is 
    \begin{equation}\label{complete-koszul-eq}
        \widehat{\mathrm{Kos}(\F; h_1, \ldots, h_n)}_{\I} = \R\varprojlim_r \big( \mathrm{Kos}(\F; h^r_1, \ldots, h^r_n) \otimes^{\mathbf{L}}_A \mathrm{Kos}(\F; f_1^r, \ldots, f_n^r)\big) 
    \end{equation}
    and the degree zero homology of the completion is 
    \[
        H^0\left(\widehat{\mathrm{Kos}(\F; h_1, \ldots, h_n)}_{\I}\right) \simeq \faktor{\widehat{\F}_{\I}}{\mathcal{J} \widehat{\F}_{\I}}.
    \]    
    Now note that for a derived tensor product where at least one factor is chain homotopic to zero, the entire product has a chain mapping to the zero object. Suppose $\mathcal{J}$ contains a regular sequence. We now have in \eqref{complete-koszul-eq} the derived limit of the zero object, to which we apply the fact that the derived limit of the zero object is homologicaly trivial.
\end{proof}

\begin{prop}
    The complex \eqref{complete-koszul-eq} is isomorphic to the the Koszul complex of the completion of $\F$, and if a sequence is regular in $\F$, then its image is regular in $\widehat{\F}$. 
\end{prop}
\begin{proof}
    By the arguments in \cite[\S 2]{SHAUL2021107806} we have 
    \[
        \widehat{\mathrm{Kos}(\F; h_1, \ldots, h_n)}_{\I} \simeq \mathrm{Kos}(\widehat{\F}_{\I}; \widehat{h_1}, \ldots, \widehat{h_n})
    \]
    so that the homology is identical, and in particular, the degree zero homology of the Koszul complex of the completion is also $\widehat{\F}_\I / \mathcal{J}\widehat{\F}_{\I}$. Now the second part of the claim follows from the isomorphism: whenever $h_1, \ldots, h_n$ is regular, $\mathrm{Kos}(\widehat{\F}_{\I}; \widehat{h_1}, \ldots, \widehat{h_n})$ is exact.
\end{proof}

\begin{theorem}\label{derived-spencer-thm} 
On a smooth scheme $X$, the derived completion of the Spencer complex is
\phantom{}
    \vskip0.5em
    \begin{thmnum}
    \setlength\itemsep{0.5em}
    \item the derived completed Spencer cohomology \label{cderived-spencer-thm-1}
    \item  stalk-wise the Spencer complex of the derived completion of a $\D_X$-module $\M^{\sbullet}_X$ \label{derived-spencer-thm-2}
    \item an acyclic resolution of $\O_{\widehat{X}}$ for $\M^{\sbullet}_X = \D_X$. \label{derived-spencer-thm-3}
    \end{thmnum}
    \vskip0.5em
\end{theorem}
\begin{proof}
    \phantom{}
    \vskip0.5em
    \begin{enumerate}
    \setlength\itemsep{0.5em}
    \item We can use the fact that 
    \[
        H^\ast(X, \widehat{\M_X^{\sbullet}}) = \mathbf{H}(\R\Gamma(X, \widehat{\M_X^{\sbullet}}))
    \]
    and Proposition \ref{quasi-iso-prop} to obtain
    \[
        H^\ast(X, \widehat{\M_X}^{\sbullet}) = \mathbf{H}(\R\Gamma(X, \widehat{\Sp^{\sbullet}(\M_X^{\sbullet}})).
    \]
    \item Follows from the arguments above.
    \item Also by the above, stalk-wise the degree zero homology is the derived completion $\widehat{\O_X}$. The terms in the complex are the derived completions of exterior powers of the tangent sheaf $\T_X$. Since completion is exact on finitely-generated modules over Noetherian rings and $\O_X$ and all $\wedge^i \T_X$ are coherent on a smooth scheme, it follows by our previous arguments that we have an acyclic chain complex. 
    \end{enumerate}
    \vskip0.5em
\end{proof}

Now we will cast our eyes towards the singular (imbeddable) case. As before we will take the derived completion of $f_+$. Very little needs modification: since $f_+$ is a $\D_X$-module supported on $Y$, we can apply the results previously seen.  

\begin{theorem}
    The derived completion of the pushforward to the point computes derived completed Spencer cohomology. When $\M_X$ is the image of some $\M_Y$ under $f_\natural$ this is acyclic.
\end{theorem}
\begin{proof}
    Take the derived completion with respect to the ideal of definition of $Y$. Since derived pushforwards commute with derived completion and we have
    \[
        \R f_*\left( \M^{\sbullet}_X \otimes^{\mathbf{L}}_{\D_X} \D_{X \to \Spec k} \right) = \R f_* \Sp^{\sbullet}(\M_X^{\sbullet}),
    \]
    it is implied by the resolution in Proposition \ref{quasi-iso-prop} that
    \[
        \widehat{\R f_*\left( \M^{\sbullet}_X \otimes^{\mathbf{L}}_{\D_X} \D_{X \to \Spec k} \right)} = \R f_* \widehat{\Sp^{\sbullet}(\M_X^{\sbullet})}.
    \]
    This proves the first claim. By Theorem \ref{derived-spencer-thm-2} this is stalk-wise equivalent to computing
    \[
        \R f_* \Sp^{\sbullet}(\widehat{\M_X^{\sbullet})}.
    \]
    Since (see \cite{tarrio1997local}) the derived $I$-adic completion of $\M^{\sbullet}_X$ satisfies
    \[
        \widehat{\M^{\sbullet}_X} \simeq \R\Hom(\R\Gamma_Y(X, \D_X), \M^{\sbullet}_X)
    \]
    we can interpret it as computing sections along the formal neighbourhood (in the $\D_X$-module sense; namely, allowing derivatives along the support). The cohomology of a derived completed $\M_X^{\sbullet}$ measures obstructions to extending sections along $Y$, hence, how much a complex of $\D_X$-modules $\M_X^{\sbullet}$ fails to be acyclic with respect to $\Gamma_Y$. If $\M_X^{\sbullet}$ is set-theoretically supported on $Y$, the higher derived completed cohomology vanishes. It follows that derived completed Spencer cohomology is acyclic for pushforwards by closed imbeddings.
\end{proof}

For this reason a modified version of Vinogradov's conjecture is true. This final corollary will follow from the theorem: 

\begin{corollary}
    The derived completed algebraic Spencer complex is acyclic if and only if a singular algebraic variety admits an imbedding into a smooth variety. 
\end{corollary}

\subsection*{Acknowledgements}

The authors thank Devlin Mallory for pointing out a gap in an intial version of the proof of Theorem \ref{only-if-thm}. DARS acknowledges support from the Einstein Chair programme at the Graduate Centre of the City University of New York. BB was supported by a studentship funded by Softmax AI. 

\bibliographystyle{myamsalpha}
\bibliography{main}

\providecommand{\bysame}{\leavevmode\hbox to3em{\hrulefill}\thinspace}
\providecommand{\MR}{\relax\ifhmode\unskip\space\fi MR }
\providecommand{\MRhref}[2]{%
  \href{http://www.ams.org/mathscinet-getitem?mr=#1}{#2}
}
\providecommand{\href}[2]{#2}
\begin{thebibliography}{BGG72}

\bibitem[BD91]{beilinson1991quantization}
Alexander Beilinson and Vladimir Drinfel'd, \emph{Quantization of {H}itchin’s integrable system and {H}ecke eigensheaves}, 1991.

\bibitem[BGG72]{bernvstein1972differential}
Joseph~H Bern{\v{s}}te{\u\i}n, Israel~M Gel’fand, and Sergei~I Gel’fand, \emph{Differential operators on a cubic cone}, Uspekhi Matematicheskikh Nauk \textbf{27} (1972), no.~1(163), 185--190 (Russian), translation in Russian Mathematical Surveys \textbf{27} (1972), no. 1, 169--174.

\bibitem[BH98]{bruns1998cohen}
Winfried Bruns and H~J{\"u}rgen Herzog, \emph{{C}ohen--{M}acaulay rings}, Cambridge Studies in Advanced Mathematics, vol.~39, Cambridge University Press, 1998.

\bibitem[Bor87]{borel}
Armand Borel, \emph{Algebraic {D}-modules}, Perspectives in Mathematics, vol.~2, Academic Press, 1987.

\bibitem[BZN04]{ben2004cusps}
David Ben-Zvi and Thomas Nevins, \emph{Cusps and $\mathcal{D}$-modules}, Journal of the American Mathematical Society \textbf{17} (2004), no.~1, 155--179.

\bibitem[EK06]{elashvili2006lie}
Alexander Elashvili and Giorgi Khimshiashvili, \emph{Lie algebras of simple hypersurface singularities}, Journal of Lie Theory \textbf{16} (2006), no.~4, 621--649.

\bibitem[ES18]{etingof2018poisson}
Pavel Etingof and Travis Schedler, \emph{Poisson traces, {D}-modules, and symplectic resolutions}, Letters in Mathematical Physics \textbf{108} (2018), no.~3, 633--678.

\bibitem[Ful84]{fulton}
William Fulton, \emph{Intersection theory}, Ergebnisse der Mathematik und ihrer Grenzgebiete, vol.~2, Springer, 1984.

\bibitem[GM92]{greenlees1992derived}
John P~C Greenlees and J~Peter May, \emph{Derived functors of {$I$}-adic completion and local homology}, Journal of Algebra \textbf{149} (1992), no.~2, 438--453.

\bibitem[GR14]{gaitsgory2011crystals}
Dennis Gaitsgory and Nick Rozenblyum, \emph{Crystals and {D}-modules}, Pure and Applied Mathematics Quarterly \textbf{10} (2014), no.~1, 57--154.

\bibitem[EGA I]{ega-i}
Alexandre Grothendieck, \emph{{\'E}l\'ements de g\'eom\'etrie alg\'ebrique: {I.} {Le} langage des sch\'emas}, Publications Math\'ematiques de l'IH\'ES \textbf{4} (1960), 5--228.

\bibitem[Gro66]{grothendieck1966rham}
\bysame, \emph{On the de {R}ham cohomology of algebraic varieties}, Publications Math{\'e}matiques de l'IH{\'E}S \textbf{29} (1966), no.~1, 95--103.

\bibitem[EGA IV]{ega-iv}
\bysame, \emph{{\'E}l\'ements de g\'eom\'etrie alg\'ebrique: {IV.} {{\'E}tude} locale des sch\'emas et des morphismes de sch\'emas, {quatri\`eme} partie}, Publications Math\'ematiques de l'IH\'ES \textbf{32} (1967), 5--361.

\bibitem[Gro68]{grothendieck1968crystals}
\bysame, \emph{Crystals and the de {R}ham cohomology of schemes}, Dix expos{\'e}s sur la cohomologie des sch{\'e}mas, Advanced Studies in Pure Mathematics, vol.~3, North-Holland Publishing Company, 1968, pp.~306--358.

\bibitem[GS66]{guillemin1966deformation}
Victor Guillemin and Shlomo Sternberg, \emph{Deformation theory of pseudogroup structures}, Memoirs of the American Mathematical Society, no.~64, American Mathematical Society, 1966.

\bibitem[Gui65]{guillemin1965integrability}
Victor Guillemin, \emph{The integrability problem for {$G$}-structures}, Transactions of the American Mathematical Society \textbf{116} (1965), 544--560.

\bibitem[Har66]{hartshorne1966residues}
Robin Hartshorne, \emph{Residues and duality}, Lecture Notes in Mathematics, vol.~20, Springer, 1966.

\bibitem[Har67]{Hartshorne1967}
\bysame, \emph{Local cohomology}, Lecture Notes in Mathematics, vol.~41, Springer-Verlag, Berlin, Heidelberg, 1967.

\bibitem[Har75]{hartshorne1975rham}
\bysame, \emph{On the de {R}ham cohomology of algebraic varieties}, Publications Math{\'e}matiques de l'IH{\'E}S \textbf{45} (1975), 5--99.

\bibitem[Hoc78]{hochster1978some}
Melvin Hochster, \emph{Some applications of the {F}robenius in characteristic 0}, Bulletin of the American Mathematical Society \textbf{84} (1978), no.~5, 886--912.

\bibitem[HTT07]{hotta2007d}
Ryoshi Hotta, Kiyoshi Takeuchi, and Toshiyuki Tanisaki, \emph{{$D$}-modules, perverse sheaves, and representation theory}, Progress in Mathematics, vol. 236, Springer, 2007.

\bibitem[Ill75]{illusie1975report}
Luc Illusie, \emph{Report on crystalline cohomology}, Algebraic Geometry---Arcata 1974, Proceedings of Symposia in Pure Mathematics, vol.~29, American Mathematical Society, 1975, pp.~459--478.

\bibitem[Joh71]{johnson1971some}
Joseph Johnson, \emph{Some homological properties of {S}pencer's cohomology theory}, Journal of Differential Geometry \textbf{5} (1971), no.~3-4, 341--351.

\bibitem[Kas78]{kashiwara1978holonomic}
Masaki Kashiwara, \emph{On the holonomic systems of linear differential equations, {II}}, Inventiones mathematicae \textbf{49} (1978), no.~2, 121--135.

\bibitem[Kov13]{kovacs2013singularities}
S{\'a}ndor~J Kov\'acs, \emph{Singularities of stable varieties}, Handbook of Moduli, vol.~2, Advanced Lectures in Mathematics, no.~25, International Press, 2013, pp.~159--203.

\bibitem[LH71]{lieberman1971duality}
David Lieberman and Miguel Herrera, \emph{Duality and the de {R}ham cohomology of infinitesimal neighborhoods}, Inventiones mathematicae \textbf{13} (1971), 97--124.

\bibitem[Mal66]{malgrange1966cohomologie}
Bernard Malgrange, \emph{Cohomologie de {S}pencer: d'apr{\`e}s {Q}uillen}, Publications du S\'eminaire de math\'ematiques d'Orsay, Secrétariat mathématique d'Orsay, 1966.

\bibitem[PW22]{pol2022homotopy}
Luca Pol and Jordan Williamson, \emph{The homotopy theory of complete modules}, Journal of Algebra \textbf{594} (2022), 74--100.

\bibitem[Qui64]{quillen1964formal}
Daniel~G Quillen, \emph{Formal properties of over-determined systems of linear partial differential equations}, Ph.D. thesis, Harvard University, 1964.

\bibitem[Sha21]{SHAUL2021107806}
Liran Shaul, \emph{{K}oszul complexes over {C}ohen--{M}acaulay rings}, Advances in Mathematics \textbf{386} (2021), 107806.

\bibitem[Spe62]{spencer1962deformation}
Donald~C Spencer, \emph{Deformation of structures on manifolds defined by transitive, continuous pseudogroups, {I, II}}, Annals of Mathematics \textbf{76} (1962), no.~2, 306--445.

\bibitem[Spe69]{spencer1969overdetermined}
\bysame, \emph{Overdetermined systems of linear partial differential equations}, Bulletin of the American Mathematical Society \textbf{75} (1969), 179--239.

\bibitem[TLL97]{tarrio1997local}
Leovigildo~Alonso Tarr{\'\i}o, Ana~Jerem{\'\i}as L{\'o}pez, and Joseph Lipman, \emph{Local homology and cohomology on schemes}, Annales scientifiques de l’Ecole normale superieure \textbf{30} (1997), no.~1, 1--39.

\bibitem[TLL99]{lipman}
\bysame, \emph{Duality and flat base change on formal schemes}, Studies in Duality on Noetherian Formal Schemes and Non-Noetherian Ordinary Schemes, Contemporary Mathematics, vol. 244, American Mathematical Society, 1999, pp.~3--90.

\bibitem[Vin79]{vinogradov1979some}
Alexandre~M Vinogradov, \emph{Some homology systems associated with the differential calculus in commutative algebras}, Uspekhi Matematicheskikh Nauk \textbf{34} (1979), no.~6(210), 145--150 (Russian), translation in Russian Mathematical Surveys \textbf{34} (1979), no. 6, 250--255.

\bibitem[Yan22]{yang2022dmodules}
Haiping Yang, \emph{{$D$}-modules on singular varieties and {H}ochschild homology of quantisations}, Ph.D. thesis, Imperial College London, 2022.

\end{thebibliography}

\end{document}